\documentclass[reqno,A4paper]{amsart}
\usepackage{amsmath,amsthm,amssymb,amsfonts,amscd}
\usepackage{enumerate}
\usepackage[noadjust]{cite}
\usepackage{diagmac2} 
\usepackage{tikz}
\usetikzlibrary{positioning,calc,arrows,arrows.meta}

\usepackage{eqlist}
\usepackage{array}

\usepackage{color}

\newtheorem{theorem}{Theorem}[section]
\newtheorem{corollary}[theorem]{Corollary}

\newtheorem{proposition}[theorem]{Proposition}

\theoremstyle{definition}
\newtheorem{definition}[theorem]{Definition}
\newtheorem{example}[theorem]{Example}

\theoremstyle{remark}
\newtheorem{remark}[theorem]{Remark}

\numberwithin{equation}{section}

\newcommand{\RR}{\mathbb{R}}

\newbox\onebox
\newcommand{\coherent}[1]{\mathbin{\setbox\onebox=\hbox{$=$}\lower0.7\ht%
\onebox\hbox{$\stackrel{#1}{=}$}}}

\begin{document}

\title[On ultrametric-preserving functions]{On ultrametric-preserving functions}
\author{Oleksiy Dovgoshey}

\newcommand{\acr}{\newline\indent}
\address{Institute of Applied Mathematics and Mechanics of NASU\acr
Dobrovolskogo str. 1, Slovyansk 84100, Ukraine}

\email{oleksiy.dovgoshey@gmail.com}

\subjclass[2010]{54E35}
\keywords{ultrametric, pseudoultrametric, pseudometric, ultrametric-preserving function.}

\begin{abstract}
Characterizations of pseudoultrametric-preserving functions and semimetric-preserving functions are found. The structural properties of pseudoultrametrics which can be represented as a composition of an ultrametric and ultrametric-pseudoultrametric-preserving function are found. A dual form of Pongsriiam--Termwuttipong characterization of the ultrametric-preserving functions is described. We also introduce a concept of \(k\)-separating family of functions and use it to characterize the ultrametric spaces.
\end{abstract}

\maketitle

\section{Introduction}

Recall some definitions from the theory of metric spaces. In what follows we write \(\RR^{+} := [0, \infty)\) for the set of all nonnegative real numbers.

\begin{definition}\label{d1.1}
A \textit{metric} on a set \(X\) is a function \(d\colon X\times X\rightarrow \RR^{+}\) such that for all \(x\), \(y\), \(z \in X\):
\begin{enumerate}
\item \((d(x,y) = 0) \Leftrightarrow (x=y)\);
\item \(d(x,y)=d(y,x)\);
\item \(d(x, y)\leq d(x, z) + d(z, y)\).
\end{enumerate}
A metric \(d\colon X\times X\rightarrow \RR^{+}\) is an ultrametric on \(X\) if
\begin{enumerate}
\item [\((iv)\)] \(d(x,y) \leq \max \{d(x,z),d(z,y)\}\)
\end{enumerate}
holds for all \(x\), \(y\), \(z \in X\).
\end{definition}

By analogy with triangle inequality \((iii)\), inequality \((iv)\) is often called the {\it strong triangle inequality}.

The theory of ultrametric spaces is closely connected with various directions of investigations in mathematics, physics, linguistics, psychology and computer science. Different properties of ultrametric spaces have been studied in~\cite{DM2009, DD2010, DP2013SM, Groot1956, Lemin1984FAA, Lemin1984RMS39:5, Lemin1984RMS39:1, Lemin1985SMD32:3, Lemin1988, Lemin2003, QD2009, QD2014, BS2017, DM2008, DLPS2008, KS2012, Vaughan1999, Vestfrid1994, Ibragimov2012, GomoryHu(1961), PTAbAppAn2014}. Note that the use of trees and tree-like structures gives a natural language for description of ultrametric spaces \cite{Carlsson2010, DLW, Fie, GurVyal(2012), Hol, H04, BH2, Lemin2003, Bestvina2002, DDP(P-adic), DP2019, DPT(Howrigid), PD(UMB), DPT2015, P2018(p-Adic),DP2018}.

The present paper is mainly motivated by characterization of ultrametric-preserving functions recently obtained by P.~Ponsgriiam and I.~Termwittipong \cite{PTAbAppAn2014}. The metric-preserving functions were detailed studied by J.~Dobo\v{s} and other mathematicians \cite{Wilson1935, Borsik1981, Borsik1988, Corazza1999, Das1989, Dobos1994, Dobos1995, Dobos1996, Dobos1996a, Dobos1998, Piotrowski2003, Pokorny1993, Termwuttipong2005, Vallin2000, Khemar2007, Bernig2003, Foertsch2002, Herburt1991, DM2013, DPK2014MS} but the properties of functions which preserve special type metrics or generalized metrics remain little studied (see \cite{KP2018MS} and \cite{KPS2019IJoMaCS} only for results related to metric-preserving functions and \(b\)-metrics). In this regard, we note that Ponsgriiam--Termwittipong characterization of ultrametric-preserving functions can be extended to characterizations of functions which preserve pseudoultrametrics, semimetrics and some other generalized metrics. Detection and description of such characteristic properties is the main goal of the paper. The pseudometric-preserving and the ultrametric-pseudoultrametric preserving functions are characterized in Proposition~\ref{ch2:p4}. A constructive characteristic of pseudoultrametric spaces which can be obtained from ultrametric spaces by using of ultrametric-pseudoultrametric-preserving functions is given in Proposition~\ref{p3.3}. Using the description of ultrametric-metric-preserving functions from~\cite{PTAbAppAn2014} we obtain also a new characteristic property of ultrametric spaces in Theorem~\ref{ch2:th3} and it is one of the main results of the paper.

\section{Ultrametrics, Pseudoultrametrics and Semimetrics}

The useful generalization of the concept of metric (ultrametric) is the concept of pseudometric (pseudoultrametric).

\begin{definition}\label{ch2:d2}
Let \(X\) be a set and let \(d \colon X \times X \to \RR^{+}\) be a symmetric function such that \(d(x, x) = 0\) holds for every \(x \in X\). The function \(d\) is a \emph{pseudometric} (\emph{pseudoultrametric}) on \(X\) if it satisfies the triangle inequality (the strong triangle inequality).
\end{definition}

If \(d\) is a pseudometric (pseudoultrametric) on \(X\), then we will say that \((X, d)\) is a \emph{pseudometric} (\emph{pseudoultrametric}) \emph{space}.

Every ultrametric space is a pseudoultrametric space but not conversely. In contrast to ultrametric spaces, pseudoultrametric spaces can contain some distinct points with zero distance between them.


\begin{example}\label{ch2:ex3}
Let \(X = \{x_1, x_2, x_3\}\) and let \(d \colon X \times X \to \RR\) be symmetric and satisfy
\[
d(x_1, x_1) = d(x_2, x_2) = d(x_3, x_3) = d(x_1, x_2) = 0
\]
and 
\[
d(x_1, x_3) = d(x_2, x_3) = t, \quad t > 0.
\]
Then \(d\) is a pseudoultrametric on \(X\) but \(d\) is not an ultrametric.
\end{example}

The next definition is a modification of Definition~1 from~\cite{PTAbAppAn2014}.

\begin{definition}\label{ch2:d5}
A function \(f \colon \RR^{+} \to \RR^{+}\) is \emph{pseudoultrametric-preserving} (\emph{ultrametric-pseudoultrametric-preserving}) if \(f \circ d\) is a pseudoultrametric for every pseudoultrametric (ultrametric) space \((X, d)\).
\end{definition}

Recall that \(f \colon \RR^{+} \to \RR^{+}\) is \emph{increasing} if 
\[
(a \geqslant b) \Rightarrow (f(a) \geqslant f(b))
\]
holds for all \(a\), \(b \in \RR^{+}\).

\begin{proposition}\label{ch2:p4}
The following conditions are equivalent for every function \(f \colon \RR^{+} \to \RR^{+}\).
\begin{enumerate}
\item\label{ch2:p4:s1} \(f\) is increasing and \(f(0) = 0\) holds;
\item\label{ch2:p4:s2} \(f\) is pseudoultrametric-preserving;
\item\label{ch2:p4:s3} \(f\) is ultrametric-pseudoultrametric-preserving.
\end{enumerate}
\end{proposition}

\begin{proof}
\(\ref{ch2:p4:s1} \Rightarrow \ref{ch2:p4:s2}\). Suppose that \(f\) is increasing and \(f(0) = 0\) holds. Let \((X, d)\) be a pseudoultrametric space. Then \(f \circ d\) is nonnegative and symmetric. The equalities \(d(x, x) = 0\) and \(f(0) = 0\) imply \(f(d(x, x)) = 0\). Since \(f\) is increasing, the strong triangle inequality for \(d\) implies this inequality for \(f \circ d\). Thus \((X, f\circ d)\) is a pseudoultrametric space.

\(\ref{ch2:p4:s2} \Rightarrow \ref{ch2:p4:s3}\). This is evidently valid.

\(\ref{ch2:p4:s3} \Rightarrow \ref{ch2:p4:s1}\). Let \(f \colon \RR^{+} \to \RR^{+}\) be ultrametric-pseudoultrametric-preserving. Then \(f \circ d\) is a pseudoultrametric for every ultrametric space \((X, d)\). Thus 
\[
f(0) = f(d(x, x)) = 0
\]
holds. If \(f\) is not increasing, then there are \(a\), \(b \in \RR^{+}\) such that 
\begin{equation}\label{ch2:p4:e1}
0 < a < b \quad \text{and} \quad f(a) > f(b).
\end{equation}
Let \(X = \{x_1, x_2, x_3\}\) and let \(d\) be an \textbf{ultrametric} on \(X\) such that 
\begin{equation}\label{ch2:p4:e2}
d(x_1, x_2) = d(x_1, x_3) = b \quad \text{and} \quad d(x_2, x_3) = a.
\end{equation}
Then~\eqref{ch2:p4:e1} and~\eqref{ch2:p4:e2} imply
\[
\max\{f(d(x_1, x_2)), f(d(x_1, x_3))\} = f(b) < f(a) = f(d(x_2, x_3)).
\]
Hence, we have the inequality
\[
\max\{f(d(x_1, x_2)), f(d(x_1, x_3))\} < f(d(x_2, x_3))
\]
which contradicts the strong triangle inequality in the space \((X, d)\).
\end{proof}

\begin{proposition}\label{p3.3}
If \((X, d)\) is an ultrametric space and \(f \colon \RR^{+} \to \RR^{+}\) is an ultrametric-pseudoultrametric-preserving function, then the pseudoultrametric space \((X, f \circ d)\) is ultrametric or there is \(r_0 > 0\) such that \(f(d(x, y)) = 0\) holds whenever \(0 < d(x, y) < r_0\).

Conversely, suppose \((Y, \rho)\) is a pseudoultrametric space such that \(\rho\) is not an ultrametric and there is \(r_0 > 0\) for which \(\rho(x, y) = 0\) holds whenever \(x\), \(y \in Y\) and \(\rho(x, y) < r_0\). Then there are an ultrametric-pseudoultrametric-preserving function \(f \colon \RR^{+} \to \RR^{+}\) and an ultrametric space \((Y, d)\) such that \(\rho = f \circ d\).

\end{proposition}

\begin{proof}
Let \((X, d)\) be an ultrametric space and let \(f \colon \RR^{+} \to \RR^{+}\) be ultrametric-pseudometric-preserving. Suppose that \((X ,f\circ d)\) is not an ultrametric. Then there are some distinct \(x_1\), \(x_2 \in X\) such that \(f(d(x_1, x_2)) = 0\). Write \(r_0 = d(x_1, x_2)\). Since \(x_1 \neq x_2\) and \(d\) is an ultrametric, the inequality \(r_0 > 0\) holds. If \(r\) is an arbitrary point of \([0, r_0)\), then, using Proposition~\ref{ch2:p4}, we obtain 
\[
0 = f(0) \leqslant f(r) \leqslant f(r_0) = 0.
\]
Thus, \(f(r) = 0\) holds for every \(r \in [0, r_0)\). 

Conversely, suppose \((Y, \rho)\) is a pseudoultrametric space such that \(\rho\) is not an ultrametric and there is \(r_0 > 0\) for which \(\rho(x, y) = 0\) holds whenever \(x\), \(y \in Y\) and \(\rho(x, y) < r_0\). Write
\[
r^* = \inf \{\rho(x, y) \colon x, y \in Y \text{ and } \rho(x, y) > 0\}.
\]
Then the inequality \(r^* > 0\) holds. Let us define a function \(d \colon Y^{2} \to \RR\) as follows
\begin{equation*}
d(x, y) = \begin{cases}
\rho(x, y)  & \text{if } \rho(x, y) > 0,\\
\frac{r^*}{2} & \text{if } \rho(x, y) = 0 \text{ and } x \neq y,\\
0           & \text{if } x = y.
\end{cases}
\end{equation*}
A direct calculation shows that \(d\) is an ultrametric on \(Y\) and the equality \(\rho = f^* \circ d\) holds for \(f^* \colon \RR^{+} \to \RR^{+}\) defined as
\[
f^*(t) = \begin{cases}
0  & \text{if } t \in \left[0, \frac{1}{2}r^*\right],\\
t  & \text{if } t > \frac{1}{2}r^*.
\end{cases}
\]
Proposition~\ref{ch2:p4} implies that \(f^*\) is ultrametric-pseudoultrametric-preserving.
\end{proof}

Let \(X\) be a nonempty set and let \(d \colon X \times X \to \RR\) be nonnegative. Wilson~\cite{Wilson1931} says that \((X, d)\) is a semimetric space and \(d\) is a semimetric on \(X\) if, for all \(x\), \(y \in X\), the following conditions are satisfied:
\begin{enumerate}
\item \(d(x, y) = 0\) if and only if \(x = y\);
\item \(d(x, y) = d(y, x)\).
\end{enumerate}

The term semimetric (= semi-metric) is used mainly in general topology. Very often the semimetrics are called \emph{dissimilarities} or simply \emph{distances}. (See~\cite[p.~15]{MDED}.)


\begin{definition}\label{ch2:d6}
A function \(f \colon \RR^{+} \to \RR^{+}\) is \emph{semimetric-preserving} if \(f \circ d\) is a semimetric for every semimetric space \((X, d)\).
\end{definition}

The function \(f \colon \RR^{+} \to \RR^{+}\) is said to be \emph{amenable} if \(f^{-1}(\{0\}) = \{0\}\).

\begin{proposition}\label{ch2:p6}
The following conditions are equivalent for every function \(f \colon \RR^{+} \to \RR^{+}\).
\begin{enumerate}
\item \label{ch2:p6:s1} \(f\) is semimetric-preserving.
\item \label{ch2:p6:s2} \(f\) is amenable.
\end{enumerate}
\end{proposition}

\begin{proof}
\(\ref{ch2:p6:s1} \Rightarrow \ref{ch2:p6:s2}\). Let us prove the truth of \(\rceil \ref{ch2:p6:s2} \Rightarrow \rceil \ref{ch2:p6:s1}\), where \(\rceil\) is the negation symbol. If \(f\) is not amenable, then
\begin{equation}\label{ch2:p6:e2}
f(0) > 0
\end{equation}
or there is \(t > 0\) such that 
\begin{equation}\label{ch2:p6:e1}
f(t) = 0
\end{equation}
holds. Let \(X = \{x_1, x_2\}\) and let 
\[
d(x_1, x_1) = d(x_2, x_2) = 0 \quad \text{and} \quad d(x_1, x_2) = t.
\]
Then \(d\) is an \textbf{ultrametric} on \(X\). Equality~\eqref{ch2:p6:e1} implies \(f(d(x_1, x_2)) = 0\), similarly from~\eqref{ch2:p6:e2} it follows that
\[
f(d(x_1, x_1)) = f(d(x_2, x_2))> 0.
\]
Thus \(f \circ d\) is not semimetric-preserving.

\(\ref{ch2:p6:s2} \Rightarrow \ref{ch2:p6:s1}\). It follows directly from the definitions.
\end{proof}

\begin{definition}[{\cite{PTAbAppAn2014}}]\label{ch2:d3}
A function \(f \colon \RR^{+} \to \RR^{+}\) is \emph{ultrametric-preserving} if \(f \circ d\) is an ultrametric for every ultrametric space \((X, d)\). We also say that \(f \colon \RR^{+} \to \RR^{+}\) is \emph{ultrametric-metric-preserving} if \(f \circ d\) is a metric for every ultrametric space \((X, d)\).
\end{definition}

The following theorem as well as Theorem~\ref{ch2:th2} was obtained by P.~Pongsriiam and I.~Termwuttipong in~\cite{PTAbAppAn2014}. To make the present paper self-contained and to show how semimetric-preserving and pseudoultrametric-preserving functions can be used for investigation of ultrametric-preserving functions, we give new proofs of these theorems.

\begin{theorem}\label{ch2:th1}
A function \(f \colon \RR^+ \to \RR^+\) is ultrametric-preserving if and only if \(f\) is amenable and increasing.
\end{theorem}

\begin{proof}
Let \(f \colon \RR^+ \to \RR^+\) be amenable and increasing. Then, by Proposition~\ref{ch2:p4}, \(f\) is pseudoultrametric-preserving and, by Proposition~\ref{ch2:p6}, \(f\) is semimetric-preserving. It is easy to see that, for every nonempty set \(X\) and every function \(d \colon X^{2} \to \RR\), \(d\) is an ultrametric on \(X\) if and only if \(d\) is simultaneously a pseudoultrametric on \(X\) and a semimetric on \(X\). Hence, \(f\) is ultrametric-preserving.

Now let \(f\) be ultrametric-preserving. If \(f\) is not increasing, then there is an \textbf{ultrametric} \(d\) such that \(f \circ d\) is not a pseudoultrametric (see the proof of Proposition~\ref{ch2:p4}). Similarly, if \(f\) is not amenable, then we can find an \textbf{ultrametric} \(d\) for which \(f \circ d\) is not a semimetric (see the proof of Proposition~\ref{ch2:p6}). This complete the proof.
\end{proof}

\begin{example}\label{ch2:ex4}
Let \((X, d)\) be an ultrametric space and let \(t \in (0, \infty)\). The function \(f \colon \RR^+ \to \RR^+\) with
\[
f(x) = \begin{cases}
x & \text{if } x \in [0, t),\\
t & \text{if } x \in [t, \infty),
\end{cases}
\]
is amenable and increasing. By Theorem~\ref{ch2:th1}, \(f \circ d\) is an ultrametric.
\end{example}

\begin{theorem}[{\cite{PTAbAppAn2014}}]\label{ch2:th2}
Let \(f \colon \RR^{+} \to \RR^{+}\) be amenable. Then the following statements are equivalent.
\begin{enumerate}
\item\label{ch2:th2:s1} The function \(f\) is ultrametric-metric-preserving.
\item\label{ch2:th2:s2} The inequality 
\begin{equation}\label{ch2:th2:e1}
f(a) \leqslant 2f(b)
\end{equation}
holds whenever \(0 \leqslant a \leqslant b\).
\end{enumerate}
\end{theorem}

\begin{proof}
\(\ref{ch2:th2:s1} \Rightarrow \ref{ch2:th2:s2}\). Let \(f\) be ultrametric-metric-preserving and let
\begin{equation}\label{ch2:th2:e2}
0 \leqslant a \leqslant b.
\end{equation}
Note that~\eqref{ch2:th2:e1} is trivial if \(a = 0\). Suppose \(a > 0\). Then there is an ultrametric space \((X, d)\) with \(X = \{x_1, x_2, x_3\}\) and such that
\begin{equation}\label{ch2:th2:e3}
a = d(x_1, x_2) \leqslant d(x_2, x_3) = d(x_3, x_1) = b.
\end{equation}
Since \(f\) is ultrametric-metric-preserving, \((X, f\circ d)\) is a metric space. Applying the triangle inequality to the metric \(f \circ d\), we can simply prove that 
\begin{multline}\label{ch2:th2:e4}
2 \max\{f(d(x_1, x_2)), f(d(x_2, x_3)), f(d(x_3, x_1))\} \\
\leqslant f(d(x_1, x_2)) + f(d(x_2, x_3)) + f(d(x_3, x_1)).
\end{multline}
The last inequality is equivalent to
\begin{equation}\label{ch2:th2:e5}
2 \max\{f(a), f(b)\} \leqslant f(a) + 2f(b).
\end{equation}
Inequality~\eqref{ch2:th2:e2} and the trivial inequality 
\[
2 f(a) \leqslant 2 \max\{f(a), f(b)\}
\]
imply inequality~\eqref{ch2:th2:e1}.

\(\ref{ch2:th2:s2} \Rightarrow \ref{ch2:th2:s1}\). Let \ref{ch2:th2:s2} hold. Let us consider an arbitrary nonempty ultrametric space \((X, d)\). We prove that \((X, f \circ d)\) is a metric space. Since \(f\) is amenable and nonnegative, it suffices to show that the triangle inequality holds for \(f \circ d\). The triangle inequality holds for \(f \circ d\) if and only if we have inequality~\eqref{ch2:th2:e4} for arbitrary triple \(x_1\), \(x_2\), \(x_3 \in X\). Since \(d\) is an ultrametric, for given \(x_1\), \(x_2\), \(x_3 \in X\), there are \(a\), \(b \in \RR^{+}\) such that~\eqref{ch2:th2:e3} and~\eqref{ch2:th2:e2} hold. Consequently it suffices to prove that~\eqref{ch2:th2:e5} holds whenever we have~\eqref{ch2:th2:e1}, \eqref{ch2:th2:e2} and\eqref{ch2:th2:e3}. Inequality~\eqref{ch2:th2:e5} is trivial if 
\[
\max \{f(a), f(b)\} = f(b).
\]
To complete the proof it suffices to note that if
\[
\max \{f(a), f(b)\} = f(a),
\]
then~\eqref{ch2:th2:e5} is equivalent to~\eqref{ch2:th2:e1}.
\end{proof}

The following characterization of ultrametrics is, in fact, dual to Theorem~\ref{ch2:th1} and Theorem~\ref{ch2:th2}.

\begin{theorem}\label{ch2:th3}
Let \((X, d)\) be a metric space. Then the following statements are equivalent.
\begin{enumerate}
\item\label{ch2:th3:s1} \((X, d)\) is an ultrametric space.
\item\label{ch2:th3:s2} \((X, f \circ d)\) is a metric space for every amenable and increasing function \(f \colon \RR^{+} \to \RR^{+}\).
\item\label{ch2:th3:s3} \((X, f \circ d)\) is a metric space for every amenable function \(f \colon \RR^{+} \to \RR^{+}\) which satisfies the inequality
\[
f(a) \leqslant 2f(b)
\]
whenever \(0 \leqslant a \leqslant b\).
\end{enumerate}
\end{theorem}

\begin{proof}
The implication \(\ref{ch2:th3:s1} \Rightarrow \ref{ch2:th3:s3}\) follows from Theorem~\ref{ch2:th2}.

It is clear that \(f(a) \leqslant 2 f(b)\) holds for every increasing function \(f \colon \RR^{+} \to \RR^{+}\) and all \(a\), \(b \in \RR^{+}\) with \(a \leqslant b\). Consequently \ref{ch2:th3:s3} implies \ref{ch2:th3:s2}.

\(\ref{ch2:th3:s2} \Rightarrow \ref{ch2:th3:s1}\). Let \ref{ch2:th3:s2} hold. If \(d\) is not an ultrametric, then we can find distinct points \(x_1\), \(x_2\), \(x_3 \in X\) such that 
\begin{equation}\label{ch2:th3:e1}
d(x_1, x_2) > \max\{d(x_1, x_3), d(x_2, x_3)\}.
\end{equation}
Write
\[
a:= \max\{d(x_1, x_3), d(x_2, x_3)\} \quad \text{and}\quad b:= d(x_1, x_2)
\]
and consider \(f_{a,b} \colon \RR^{+} \to \RR^{+}\) such that
\begin{equation}\label{ch2:th3:e3}
f_{a,b}(t) = \begin{cases}
0 & \text{if } t = 0,\\
\frac{1}{2} a & \text{if } t \in (0, a],\\
b & \text{if } t \in (a, \infty).
\end{cases}
\end{equation}
Then \(f_{a,b}\) is increasing and amenable. It follows from the definition of \(f_{a,b}\) and~\eqref{ch2:th3:e1} that
\begin{equation}\label{ch2:th3:e2}
f_{a,b}(d(x_1, x_2)) = f_{a,b}(b) = b \text{ and } f_{a,b}(d(x_1, x_3)) = f_{a,b}(d(x_2, x_3)) = \frac{1}{2} a.
\end{equation}
By statement~\ref{ch2:th3:s2}, \(f_{a,b} \circ d\) is a metric on \(X\). Now using~\eqref{ch2:th3:e2} and the triangle inequality we obtain
\[
b = f_{a,b}(d(x_1, x_2)) \leqslant f_{a,b}(d(x_1, x_3)) + f_{a,b}(d(x_2, x_3)) = \frac{1}{2} a + \frac{1}{2} a = a.
\]
Thus \(b \leqslant a\) contrary to~\eqref{ch2:th3:e1}.
\end{proof}

Analyzing the proof of Theorem~\ref{ch2:th3} we obtain the following corollary.

\begin{corollary}\label{ch2:c1}
Let \((X, d)\) be a metric space. Then \(d\) is an ultrametric if and only if \(f_{a,b} \circ d\) is a metric for all strictly positive \(a\) and \(b\), where \(f_{a, b}\) is defined by~\eqref{ch2:th3:e3}.
\end{corollary}

Let \(A\) and \(B\) be two sets and let \(\mathcal{F}\) be a family of mappings from \(A\) to \(B\). Recall that \(\mathcal{F}\) is said to \emph{separate points} on \(A\) if for every two distinct \(a_1\), \(a_2 \in A\) there is \(f \in \mathcal{F}\) such that \(f(a_1)\neq f(a_2)\) (see, for example, \cite{Rudin1976}, Definition~7.30).

\begin{definition}\label{ch2:d1}
Let \(\mathcal{F}\) be a set of increasing and amenable functions \(f \colon \RR^{+} \to \RR^{+}\) and let \(k \in (1, \infty)\). Then \(\mathcal{F}\) is \(k\)-separating if for every two \(t_1\), \(t_2 \in \RR^{+}\) with \(t_1 < t_2\), there is \(f \in \mathcal{F}\) such that \(k f(t_1) < f(t_2)\).
\end{definition}

\begin{theorem}\label{ch2:th4}
Let \(\mathcal{F}\) be a set of increasing and amenable functions \(f \colon \RR^{+} \to \RR^{+}\). If \(\mathcal{F}\) is \(2\)-separating, then the following statements are equivalent for every metric space \((X, d)\):
\begin{enumerate}
\item\label{ch2:th4:s1} For every \(f \in \mathcal{F}\) the function \(f \circ d\) is a metric on \(X\);
\item\label{ch2:th4:s2} \((X, d)\) is an ultrametric space.
\end{enumerate}
If \(\mathcal{F}\) is not \(2\)-separating, then there is a metric space \((X, d)\) such that \(f \circ d\) is a metric on \(X\) for every \(f \in \mathcal{F}\), but \(d\) is not an ultrametric on \(X\).
\end{theorem}

\begin{proof}
Let \(\mathcal{F}\) be \(2\)-separating and let \((X, d)\) be a metric space. 

\(\ref{ch2:th4:s1} \Rightarrow \ref{ch2:th4:s2}\). Suppose \(f \circ d\) is a metric for every \(f \in \mathcal{F}\). If \(d\) is not an ultrametric on \(X\), then there exist \(x_1\), \(x_2\), \(x_3 \in X\) such that
\[
a := \max \{d(x_1, x_3), d(x_2, x_3)\}, \quad \text{and} \quad b := d(x_1, x_2), \quad \text{and} \quad b > a > 0.
\]
Since \(\mathcal{F}\) is \(2\)-separating, there is \(f \in \mathcal{F}\) such that
\begin{equation}\label{ch2:th4:e4}
2f(a) < f(b).
\end{equation}
The function \(f\) is increasing. It implies the inequalities
\begin{equation}\label{ch2:th4:e5}
f(d(x_1, x_3)) \leqslant f(a) \quad\text{and} \quad f(d(x_2, x_3)) \leqslant f(a).
\end{equation}
From~\eqref{ch2:th4:e4} and \eqref{ch2:th4:e5} we obtain
\[
f(d(x_1, x_3)) + f(d(x_2, x_3)) < f(b) = f(d(x_1, x_2)),
\]
contrary to the triangle inequality for the metric \(f \circ d\).

\(\ref{ch2:th4:s2} \Rightarrow \ref{ch2:th4:s1}\). The validity of this implication follows from Theorem~\ref{ch2:th3}.

Suppose now that \(\mathcal{F}\) is not \(2\)-separating. Then it follows from Definition~\ref{ch2:d1} that there are \(t_1\), \(t_2 \in \RR^{+}\) such that \(0 < t_1 < t_2\) and
\begin{equation}\label{ch2:th4:e1}
2 f(t_1) \geqslant f(t_2)
\end{equation}
for all \(f \in \mathcal{F}\). We can find \(t_3 \in (t_1, t_2)\) such that the inequality 
\begin{equation}\label{ch2:th4:e2}
2 t_3 \geqslant t_2
\end{equation}
holds. Since every \(f \in \mathcal{F}\) is increasing, the condition \(t_3 \in (t_1, t_2)\) and~\eqref{ch2:th4:e1} imply the inequality
\begin{equation}\label{ch2:th4:e3}
2 f(t_3) \geqslant f(t_2)
\end{equation}
for every \(f \in \mathcal{F}\). Let \(X = \{x_1, x_2, x_3\}\) and let \(d \colon X \times X \to \RR\) be a function such that \(d(x_1, x_1) = d(x_2, x_2) = d(x_3, x_3) = 0\) and \(d(x_1, x_2) = t_2\) and \(t_3 = d(x_1, x_3) = d(x_2, x_3)\). Then we have 
\begin{equation*}
\max\{d(x_1, x_2), d(x_1, x_3), d(x_2, x_3)\} = d(x_1, x_2) = t_2 \leqslant 2t_3 = d(x_1, x_3) + d(x_2, x_3).
\end{equation*}
Hence, \((X, d)\) is a metric space. From~\eqref{ch2:th4:e2} and the definition of \(d\) it follows that \((X, d)\) is not an ultrametric space. To complete the proof it suffices to note that~\eqref{ch2:th4:e3} implies the triangle inequality for \(f \circ d\) with every \(f \in \mathcal{F}\).
\end{proof}

\begin{example}\label{ch2:ex1}
For every \(\alpha \in (0, \infty)\) and every \(t \in \RR^+\), we write 
\begin{equation}\label{ch2:ex1:e1}
f_{\alpha}(t) = t^{\alpha}.
\end{equation}
It is clear that every function \(f_{\alpha} \colon \RR^+ \to \RR^+\) defined by~\eqref{ch2:ex1:e1} is increasing and amenable. Let \(\mathcal{F} := \{f_{\alpha} \colon \alpha \in (0, \infty)\}\). The limit relation
\[
\lim_{\alpha \to \infty} \left(\frac{t_1}{t_2}\right)^{\alpha} = 0
\]
holds if \(0 \leqslant t_1 < t_2 < \infty\). Consequently \(\mathcal{F}\) is \(k\)-separating family for every \(k> 1\).
\end{example}

This example and Theorem~\ref{ch2:th4} imply the following.

\begin{corollary}\label{ch2:c2}
Let \((X, d)\) be a metric space. Then \(d\) is an ultrametric if and only if \(d^{\alpha}\) is a metric for every \(\alpha > 1\).
\end{corollary}

\begin{remark}
If \(\alpha \in (0, 1)\), then \(d^{\alpha}\) is a metric for every metric space \((X, d)\). Following paper~\cite{TW2005}, we can say that the space \((X, d^\alpha)\) is a \(\frac{1}{\alpha}\)-snowflake. Thus Corollary~\ref{ch2:c2} claims that a metric \(d\) is an ultrametric if and only if \(d\) is \(\frac{1}{\alpha}\)-snowflake for every \(\alpha \in (0, 1)\). The proof of ultrametricity of the so-called metric space of resistances given by V.~Gurvich and A.~Gvishiani in \cite{Gurvich2010, Gvishiani1987} is a nontrivial example of application of the snowflake transformation \(d \mapsto d^{\alpha}\) in real-world model.
\end{remark}



\end{document}